\documentclass[11pt]{amsart}

\usepackage{amssymb,mathrsfs,enumitem,amsmath,amsthm,bbold}
\usepackage[mathscr]{euscript}
\usepackage{color}

\usepackage[oldstyle]{libertine}

\usepackage[all]{xy}
\definecolor{Chocolat}{rgb}{0.36, 0.2, 0.09}
\definecolor{BleuTresFonce}{rgb}{0.215, 0.215, 0.36}
\definecolor{EgyptianBlue}{rgb}{0.06, 0.2, 0.65}

\usepackage[OT2,OT1]{fontenc}

\usepackage[backend=bibtex,
    sorting=anyt,
    isbn=false,
    url=false,
    doi=false,
    maxbibnames=100
    ]{biblatex}
\usepackage[colorlinks,final,hyperindex]{hyperref}
\hypersetup{citecolor=BleuTresFonce, urlcolor=EgyptianBlue, linkcolor=Chocolat}

\newtheorem{theorem}{Theorem}[section]
\newtheorem{corollary}[theorem]{Corollary}

\newtheorem{proposition}[theorem]{Proposition}

\theoremstyle{definition}

\newtheorem{remark}[theorem]{Remark}

\DeclareMathAlphabet{\pazocal}{OMS}{zplm}{m}{n}

\def\calA{\pazocal{A}}

\def\calF{\pazocal{F}}

\def\calL{\pazocal{L}}
\def\calM{\pazocal{M}}

\def\calO{\pazocal{O}}

\def\calR{\pazocal{R}}
\def\calS{\pazocal{S}}
\def\calT{\pazocal{T}}
\def\calU{\pazocal{U}}
\def\calV{\pazocal{V}}

\def\m{{\ensuremath{\mathsf{m}}}}
\def\o{{\ensuremath{\mathsf{o}}}}

\def\w{{\ensuremath{\mathsf{w}}}}

\DeclareMathOperator{\RT}{RT}
\DeclareMathOperator{\RTM}{RTM}
\DeclareMathOperator{\RTW}{RTW}
\DeclareMathOperator{\PL}{PL}
\DeclareMathOperator{\PLM}{PLM}
\DeclareMathOperator{\PLMC}{PLMC}

\DeclareMathOperator{\Graphs}{Graphs}
\DeclareMathOperator{\CGraphs}{CGraphs}

\def\ract{\mathbin{\blacktriangleleft}}

\DeclareMathOperator{\Lie}{Lie}

\DeclareMathOperator{\Com}{Com}

\DeclareMathOperator{\Perm}{Perm}

\DeclareMathOperator{\id}{id}

\DeclareMathOperator{\Div}{div}

\DeclareMathOperator{\Cyc}{Cyc}

\DeclareMathOperator{\Tw}{Tw}

\DeclareMathOperator{\Vect}{{\ensuremath\mathsf{Vect}}}

\def\d{\mathrm{d}}

\DeclareMathAlphabet{\mathbbold}{U}{bbold}{m}{n}

\def\k{\mathbbold{k}}

\begin{document}

\title{Volume preservation of Butcher series methods from the operad viewpoint}

\author{Vladimir Dotsenko}
\address{ 
Institut de Recherche Math\'ematique Avanc\'ee, UMR 7501, Universit\'e de Strasbourg et CNRS, 7 rue Ren\'e-Descartes, 67000 Strasbourg, France}

\email{vdotsenko@unistra.fr}

\author{Paul Laubie}
\address{Institut Élie Cartan de Lorraine, UMR 7502, Faculté des Sciences et Technologies, Boulevard des Aiguillettes, 54506 Vandœuvre-lès-Nancy, France}
\email{paul.laubie@univ-lorraine.fr}

\begin{abstract}
We study a coloured operad involving rooted trees and directed cycles of rooted trees that generalizes the operad of rooted trees of Chapoton and Livernet. We describe all the relations between the generators of a certain suboperad of that operad, and compute the Chevalley--Eilenberg homology of two naturally arising differential graded Lie algebras. This allows us to give short and conceptual new proofs of two important results on Butcher series methods of numerical solution of ODEs: absence of volume-preserving integration schemes and the acyclicity of the aromatic bicomplex, the key step in a complete classification of volume-preserving integration schemes using the so called aromatic Butcher series. 
\end{abstract}

\maketitle

\section{Introduction}

In numerical methods of solving ordinary differential equations, Butcher series methods introduced in \cite{MR305608,MR403225} play a prominent role. They are based on series whose terms are indexed by rooted trees; this corresponds to the fact known already to Cayley~\cite{Cayley1857} that rooted trees are a convenient book-keeping device for organising computations with iterated applications of a vector field to a function. If one considers source-free ODEs, the question of volume preservation~\cite{MR2334044} becomes natural. A general theorem proved independently by Chartier and Murua \cite{MR2317009} and by Iserles, Quispel, and Tse \cite{MR2334044} asserts that a nontrivial Butcher series method cannot be volume-preserving. By the very nature of the question of volume preservation, one wishes to keep track of divergences, and the so called aromatic Butcher series \cite{MR3451427} become relevant. Combinatorially, this amounts to considering rooted trees with coefficients that are (disjoint unions of) directed cycles of rooted trees; the latter are called aromas in the literature, hence the terminology. In fact, if one uses aromatic Butcher series, volume-preserving methods exist, and the acyclicity theorem for the so called aromatic bicomplex proved by Laurent, McLachlan, Munthe-Kaas, and Verdier \cite{MR4624837} offers a full classification of volume-preserving aromatic Butcher series methods. 

Our work offers algebraic and categorical insight into matters of volume preservation. Recall that Chapoton and Livernet \cite{MR1827084} equipped the species of labelled rooted trees $\RT$ with an operad structure and proved that the rooted tree operad is isomorphic to the operad of pre-Lie algebras, known also as right-symmetric algebras. In this paper, we study a two-coloured operad that we denote $\RTW$; it includes both rooted trees and directed cycles of rooted trees and is therefore suitable for working with aromatic Butcher series. By contrast to the operad of rooted trees, the coloured operad $\RTW$ is not finitely generated. However, we managed to describe by generators and relations its suboperad generated by the main operations of interest. This description can be used to give a conceptual (and very short) proof of the abovementioned theorem on absence of volume-preserving Butcher series methods. Moreover, we relate the aromatic bicomplex and its divergence-free version to the Chevalley--Eilenberg complexes of certain differential graded Lie algebras, and use that relationship to obtain conceptual proofs of the acyclicity theorems. In a way, one can think of our approach as follows. The apparatus of Butcher series is designed to capture universal phenomena that do not use specific features of a concrete problem (dimension, extra symmetries etc.); we take that universality to the next level and consider an object (an operad) that additionally makes it possible to break a vector field in an unspecified number of pieces that may play different roles in the integration scheme. 

\subsection{Structure of the paper}In Section \ref{sec:recollection}, we give a short recollection of topics related to volume preservation, namely, Butcher series methods, their aromatic version, and the aromatic bicomplex. In Section \ref{sec:colouredops}, we define the coloured operad $\RTW$ that is used throughout the paper. In Section \ref{sec:embedding}, we study the suboperad of $\RTW$ generated by its fundamental operations, and identify the complete set of relations between the generators. In Section \ref{sec:divergence}, we define the reduced divergence map, use our results to describe its kernel and apply that description to prove that a nontrivial Butcher series method cannot be volume-preserving. In Section \ref{sec:dgla}, we define two differential graded Lie algebras, compute their Chevalley--Eilenberg homology, and apply the thus obtained results to the aromatic bicomplex. Finally, in Appendix~\ref{app:conseqdiv0} we use the kernel of the reduced divergence map to give a new characterization of Lie elements in free pre-Lie algebras, and discuss a connection to operadic twisting, in Appendix~\ref{app:graphcx}, we prove two results on undirected graph complexes that are counterparts of our main theorems, and in Appendix~\ref{app:colouredGB} we explain how one can study one of the main operads considered here using Gr\"obner bases for coloured operads. 

\subsection{Conventions} We represent both rooted trees and directed cycles of rooted trees as directed graphs: for each rooted tree, all edges are directed towards the root. All vector spaces and chain complexes in this paper are defined over a ground field $\k$ of zero characteristic. 
For a systematic introduction to Butcher series, we refer the reader to \cite[Chapter~III]{MR1904823} and \cite[Chapter~II]{MR868663}; we also recommend the beautiful survey paper \cite{MR3751443} for some historical context and further references. We freely use terminology of species of structures \cite{MR1629341,MR633783}. For information on operads, the reader is referred to the monograph \cite{MR2954392}. For coloured operads, the reader is referred to the papers~\cite{MR2342815,MR4375008}. We only use two-coloured operads, and we view their components as coloured species: for a two-coloured operad~$\calO$, we denote by $\calO(I,J;c)$ the component with the output colour $c$, the set~$I$ of inputs of the first colour and the set~$J$ of inputs of the second colour. Chain complexes are always homologically graded (so that differential is of homological degree $-1$); to handle suspensions, we use the formal symbol $s$ of degree $1$ so that 
 \[
sC_\bullet=\k s\otimes C_\bullet=C_{\bullet-1}.
 \]

\section{Recollections on volume preservation}\label{sec:recollection}

Consider an autonomous differential equation 
 \[
y'(t)=f(y(t)),\quad y(0)=y_0,
 \]
where $y\in\mathbb{R}^d$ is an unknown vector function, and $f\colon\mathbb{R}^d\to\mathbb{R}^d$ is a smooth vector field, which one assumes to be Lipschitz to ensure the usual existence and uniqueness results. 

\subsection{Butcher series methods}
A \emph{Butcher series method} is an iterative solution method for computing the value $y(T)$ by setting $T=Nh$ and defining the approximations $y_n\approx y(nh)$ by a one-step recursion
 \[
y_{n+1}=y_n+\sum_{\tau\in\RT(\bullet)}\frac{h^{|V(\tau)|}}{\sigma(\tau)}a(\tau)F(\tau)(f)(y_n),
 \]
where $\RT(\bullet)$ is the set of all rooted trees with indistinguishable vertices all carrying the same label $\bullet$, $V(\tau)$ is the set of vertices of the tree $\tau$, $\sigma(\tau)$ is the order of the symmetry group of the tree $\tau$, $a\colon \RT(\bullet)\to\mathbb{R}$ is a sequence of coefficients that define the method, and $F(\tau)(f)$ is the so called elementary differential corresponding to the rooted tree $\tau$. They are defined recursively by $F(\bullet)(f)=f$ and, if the root of $\tau$ has $k$ children $\tau_1$, \ldots, $\tau_k$, 
 \[
F(\tau)(f)=f^{(k)}(F(\tau_1)(f),\ldots,F(\tau_k)(f)),
 \]
so that, for example, for $f=\sum_{a=1}^d f^a\partial_a$, we have
 \[
F\left(
\vcenter{
\xymatrix@M=3pt@R=5pt@C=5pt{
*=0{\bullet}\ar@{->}[dr] &&  *=0{\bullet}\ar@{->}[dl] \\
& *=0{\bullet} & 
}}
\right)(f)=f''(f,f)=\sum_{a,b,c=1}^d f^af^b\partial_a\partial_b(f^c)\partial_c.
 \]
It is known \cite[Sec.~II.2, Exercise 4]{MR868663} that if one wishes to consider general methods that work for arbitrary dimension $d$, the elementary differentials corresponding to different rooted trees are linearly independent.

Philosophy of the backward error analysis suggests that one thinks of the numerical solution given by the formula above as the exact solution of the modified differential equation $\tilde{y}'(t)=\tilde{f}(\tilde{y}(t))$, where the modified vector field $\tilde{f}$ is obtained from $f$ by the rule
 \[
\tilde{f}=\sum_{\tau\in\RT(\bullet)}\frac{h^{|V(\tau)|-1}}{\sigma(\tau)}b(\tau)F(\tau).
 \]
where $b\colon \RT(\bullet)\to\mathbb{R}$ can be explicitly described \cite[Th.~9.2]{MR1904823}; for our purposes, it is only important that there is a one-to-one correspondence expressing the coefficients $a(-)$ via the coefficients $b(-)$ and vice versa, so that one can focus on the modified equation and the coefficients $b(-)$ only.

Suppose that the original ODE was source-free, so that we have $\Div(f)=0$. Then for a B-series method to be volume preserving, we should have the modified equation satisfy the same condition, that is, $\Div(\tilde{f})=0$. A direct computation \cite[Sec.~3]{MR2334044} shows that 
 \[
\Div(F(\tau))=\sum_{\alpha\in \circlearrowright(\tau)}F(\alpha),
 \]
where $\alpha$ is a directed cycle of rooted trees, $\circlearrowright(\tau)$ is the set of all directed cycles of rooted trees obtained from $\tau$ by ``closing it up into a cycle'', that is, drawing an arrow from the root to some vertex, and $F(\alpha)$ is the elementary differential given by the same formula as above. Similar to the case of elementary differentials and trees, one can show \cite[Sec.~4.2]{MR2334044} that in the case of general methods that work for arbitrary dimension $d$, the elementary differentials corresponding to different directed cycles of rooted trees are linearly independent.

In the case of a source-free equation, the condition $\Div(f)=0$ together with its derivatives allows one to simplify the expression for $\Div(\tilde{f})$, suppressing all the terms whose directed cycle is of length $1$ (a simple loop); the volume preservation condition on the Butcher series method requires that all other terms in $\Div(\tilde{f})$ cancel each other. As mentioned above, it turns out \cite{MR2317009,MR2334044} that the only volume-preserving method corresponds to the trivial flow.
 
\subsection{Aromatic Butcher series methods} 

Aromatic Butcher series are formal power series of unlabelled rooted trees $\tau$ with coefficients being disjoint unions of directed cycles of unlabelled rooted trees $\alpha$. Expressions like that are referred to as \emph{aromatic trees}. Both of those have the corresponding elementary differentials~$F(-)$, and we can even define elementary differentials for the indexing set of aromatic Butcher series; for instance, we have
 \[
F\left(
\vcenter{
\xygraph{ 
!{<0cm,0cm>;<0.4cm,0cm>:<0cm,0.4cm>::} 
!~-{@{-}@[|(2.5)]}
!{(-0.5,1) }*=0{\bullet}="a"
!{(0,0) }*=0{\bullet}="c"
"c" : @`{p+(2,1),c+(-2,1)} "c"
"a" : "c"
}}
\vcenter{
\xygraph{ 
!{<0cm,0cm>;<0.4cm,0cm>:<0cm,0.4cm>::} 
!~-{@{-}@[|(2.5)]}
!{(0,0) }*=0{\bullet}="a"
"a" : @`{p+(2,1),c+(-2,1)} "a"
}}
\vcenter{
\xygraph{ 
!{<0cm,0cm>;<0.5cm,0cm>:<0cm,0.5cm>::} 
!~-{@{-}@[|(3.0)]}
!{(0,1) }*=0{\bullet}="a"
!{(0,0) }*=0{\bullet}="c"
"a" : "c"
}}
\right)=
\sum_{p,q,r,s,t=1}^d f^p\partial_p\partial_q(f^q) \partial_r(f^r) f^s \partial_s(f^t)\partial_t = f'(\Div(f)) \Div(f) f'(f).
 \]
It turns out that there exist nontrivial volume-preserving integration schemes using aromatic trees, and it is even possible to describe all volume-preserving schemes. Similarly to the case of the 3D vector calculus, where the kernel of the divergence coincides with the image of curl, it is possible to describe all aromatic Butcher series with zero divergence, if one extends the picture to higher homological degrees. This was done by Laurent, McLachlan, Munthe-Kaas, and Verdier in~\cite{MR4624837}, and we briefly recall their approach here. They construct an object that they call the \emph{aromatic bicomplex}; it is related to the variational bicomplex \cite{MR1188434,MR600611,MR667492,MR483733} for the jet bundle of the tangent bundle of $\mathbb{R}^d$, except that it is, like everything we already discussed, a universal dimension-independent version of it, or rather the universal version of the subcomplex of invariants with respect to the action of $GL_d(\mathbb{R})$. 

To define the aromatic bicomplex, one considers \emph{aromatic forests}, that is, graphs for which each connected component is either a rooted tree or a directed cycle of rooted trees. It is convenient to start with ordered aromatic forests, which means that the cycle components can be reordered arbitrarily but the tree components are given a linear order. The next step is to add labels, and consider ordered aromatic forests whose vertices are coloured black and white, so that black vertices are indistinguishable and the white vertices are distinguishable (labelled by $\{1,\ldots,k\}$, where $k$ is the total number of white vertices). We shall denote $\calF_{n,p}$ the set of all labelled ordered aromatic forests with $n$ roots and $p$ white vertices. The \emph{aromatic bicomplex} $\Omega_{\bullet,\bullet}$ is obtained from the linear span of labelled ordered aromatic forests by anti-symmetrizing them over the group of permutations of the tree components and over the group of permutations of white vertices. One denotes by $\Omega_{p,q}$ the linear span of aromatic forests with $p$ rooted tree components and $q$ white vertices. The horizontal differential~$d^H$ of bi-degree $(-1,0)$ is the sum of all ways to adding an edge directed from the root of the last tree component to one of the vertices of our aromatic forest. The vertical differential~$d^V$ of bi-degree $(0,-1)$ is the skew-symmetrization over white vertices of the sum of all ways to replace a black vertex of an aromatic forest with a new white vertex. The \emph{divergence-free aromatic bicomplex} $\tilde\Omega_{\bullet,\bullet}$ is the quotient of $\Omega_{\bullet,\bullet}$  by the subcomplex of aromatic forests that contain at least one component with a directed cycle of length one. Volume preservation condition for aromatic Butcher series can be expressed in terms of the kernel of the horizontal differential of the aromatic bicomplex, which was the main motivation for introducing it in\cite{MR4624837}.   

\section{Rooted trees, cycles of rooted trees, and related coloured operads}\label{sec:colouredops}

Let $\RT$ be the linear species of rooted trees, so that for a finite set, the vector space $\RT(I)$ has a basis indexed by rooted trees on $I$. There is a remarkable operad structure on $\RT$ defined by Chapoton and Livernet \cite{MR1827084}. Specifically, if $T_1\in\RT(I\sqcup\{\star\})$, and $T_2\in\RT(J)$, the element 
 \[
T_1\circ_\star T_2\in\RT(I\sqcup J)
 \] 
is equal to the sum of all rooted trees $T$ for which the subgraph on the vertex set $J$ is isomorphic to $T_2$, and the result of contracting that subgraph onto a point labelled $\star$ is isomorphic to $T_1$. Equivalently, one replaces the vertex of $T_1$ labelled~$\star$ by the tree $T_2$, and connects the root to the output edge of that vertex (if any), and sums over all ways to connect the incoming edges of $\star$ in $T_1$ to vertices of $T_2$. For example, we have
 \[
\vcenter
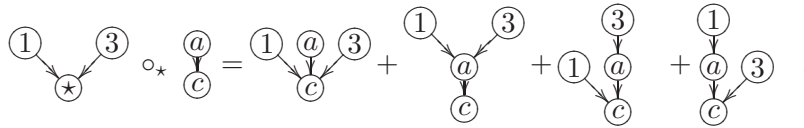
{
\xymatrix@M=3pt@R=5pt@C=5pt{
*+[o][F-]{1}\ar@{->}[dr] &&  *+[o][F-]{3}\ar@{->}[dl] \\
& *+[o][F-]{\star} & 
}}
\ \circ_\star \ 
\vcenter{
\xymatrix@M=3pt@R=5pt@C=5pt{
*+[o][F-]{a}\ar@{->}[d]  \\
*+[o][F-]{c} 
}}=
\vcenter{
\xymatrix@M=3pt@R=5pt@C=5pt{
*+[o][F-]{1}\ar@{->}[dr] &*+[o][F-]{a}\ar@{->}[d] & *+[o][F-]{3}\ar@{->}[dl]\\
& *+[o][F-]{c} & 
}}+
\vcenter{
\xymatrix@M=3pt@R=5pt@C=5pt{
*+[o][F-]{1}\ar@{->}[dr] & & *+[o][F-]{3}\ar@{->}[dl]\\
 &  *+[o][F-]{a}\ar@{->}[d]& \\
& *+[o][F-]{c} & 
}}+
\vcenter{
\xymatrix@M=3pt@R=5pt@C=5pt{
  & *+[o][F-]{3}\ar@{->}[d]\\
*+[o][F-]{1}\ar@{->}[dr] &  *+[o][F-]{a}\ar@{->}[d]& \\
& *+[o][F-]{c} & 
}}+
\vcenter{
\xymatrix@M=3pt@R=5pt@C=5pt{
*+[o][F-]{1}\ar@{->}[d] & & \\
  *+[o][F-]{a}\ar@{->}[d]& *+[o][F-]{3}\ar@{->}[dl]\\
 *+[o][F-]{c} & 
}}.
 \]
It is shown in \cite{MR1827084} that the operad $\RT$ is generated by 
 \[
\vcenter{
\xymatrix@M=3pt@R=5pt@C=5pt{
*+[o][F-]{2}\ar@{->}[d] \\
*+[o][F-]{1} & 
}} \in\RT(\{1,2\})
 \]
and is isomorphic to the operad $\PL$ of pre-Lie algebras. Recall that a \emph{pre-Lie algebra} is a vector space $A$ equipped with a binary operation $\triangleleft$ satisfying the identity
 \[
(a_1\triangleleft a_2)\triangleleft a_3 - a_1\triangleleft (a_2\triangleleft a_3)=(a_1\triangleleft a_3)\triangleleft a_2 - a_1\triangleleft (a_3\triangleleft a_2) 
 \]  
for all $a_1,a_2,a_3\in A$. For future reference, we note that the operation 
 \[
[a_1,a_2]:=a_1\triangleleft a_2-a_2\triangleleft a_1
 \]
satisfies the Jacobi identity and generates a suboperad of $\RT$ isomorphic to the operad of Lie algebras. 

Willwacher in \cite{MR4519144} considered a $\{\o,\m\}$-coloured operad $\RTM$ extending $\RT$. The only non-zero components of the two-coloured species $\RTM$ are 
 \[
\RTM(I,\varnothing;\o)=\RT(I)
 \] 
and 
 \[
\RTM(I,\{a\};\m)=\RT(I;a),
 \] 
the linear span of rooted trees on the vertex set $I\sqcup\{a\}$ for which the label of the root vertex is $a$. This species has a coloured operad structure that is an immediate generalization of the Chapoton--Livernet structure on $\RT$. We shall draw rooted trees from $\RT(I;a)$ with a dashed circle around the root vertex to not confuse those trees with trees from $\RT$; for instance, we have 
 \[
\vcenter{
\xymatrix@M=3pt@R=5pt@C=5pt{
*+[o][F-]{1}\ar@{->}[d] \\
*+[o][F.]{\star} & 
}}
\ \circ_\star \ 
\vcenter{
\xymatrix@M=3pt@R=5pt@C=5pt{
*+[o][F-]{a}\ar@{->}[d]  \\
*+[o][F.]{c} 
}}=
\vcenter{
\xymatrix@M=3pt@R=5pt@C=5pt{
*+[o][F-]{1}\ar@{->}[dr]& &*+[o][F-]{a}\ar@{->}[dl] \\
& *+[o][F.]{c} & 
}}+
\vcenter{
\xymatrix@M=3pt@R=5pt@C=5pt{
*+[o][F-]{1}\ar@{->}[d] \\
*+[o][F-]{a}\ar@{->}[d]& \\
*+[o][F.]{c} & 
}} .
 \]
It turns out \cite[Th.~3.4]{MR4519144} that the coloured operad $\RTM$ is generated by its elements
 \[
\vcenter{
\xymatrix@M=3pt@R=5pt@C=5pt{
*+[o][F-]{2}\ar@{->}[d] \\
*+[o][F-]{1} & 
}}\in \RTM(\{1,2\},\varnothing;\o) \quad\text{ and }
\quad
\vcenter{
\xymatrix@M=3pt@R=5pt@C=5pt{
*+[o][F-]{2}\ar@{->}[d] \\
*+[o][F.]{1} & 
}} \in \RTM(\{1\},\{2\};\m)
 \]
and is isomorphic to the coloured operad $\PLM$ whose algebras are pairs $(A,M)$, where $A$ is a pre-Lie algebra and $M$ is a right $A$-module. In general, the notion of a one-sided module for a given type of nonassociative algebras is not well defined; however, the leftmost argument in the pre-Lie identity does not move, which makes it possible to talk about right modules. Specifically, for a pre-Lie algebra $A$, we say that a \emph{right $A$-module} structure on vector space $M$ is a map $\ract\colon M\otimes A\to A$ such that 
 \[
(m\ract a_1)\ract a_2 - m\ract (a_1\triangleleft a_2)=(m\ract a_2)\ract a_1 - m\ract (a_2\triangleleft a_1) 
 \] 
for all $m\in M$, $a_1,a_2\in A$. Imposing this identity amounts to requiring that the assignment $m\mapsto m\ract a$ defines on $M$ a right module over the universal enveloping algebra of the Lie algebra $(A,[-,-])$.

The main algebraic structure studied in this paper is a new $\{\o,\m\}$-coloured operad $\RTW$ extending $\RTM$. The only non-zero components of the two-coloured species $\RTW$ are 
 \[
\RTW(I,\varnothing;\o)=\RT(I),\quad \RTW(I,\{a\};\m)=\RTM(I,\{a\};\m),
 \] 
and
 \[
\RTW(I,\varnothing;\m)=\Cyc(\RT)(I), 
 \] 
the linear span of (non-empty) directed cycles of rooted trees on the vertex set $I$. The species $\RTW$ has a coloured operad structure that is an immediate generalization of the Chapoton--Livernet structure on $\RT$; for instance, we have
 \[
\vcenter{
\xymatrix@M=3pt@R=5pt@C=5pt{
*+[o][F-]{1}\ar@{->}[d]  \\
*+[o][F.]{\star} & 
}}
\ \circ_\star \
\vcenter{
\xygraph{ 
!{<0cm,0cm>;<0.7cm,0cm>:<0cm,0.7cm>::} 
!~-{@{-}@[|(2.5)]}
!{(-0.5,1) }*+[o][F-]{a}="a"
!{(0,0) }*+[o][F-]{c}="c"
"c" : @`{p+(2,1),c+(-2,1)} "a"
"a" : "c"
}}=
\vcenter{\xygraph{ 
!{<0cm,0cm>;<0.7cm,0cm>:<0cm,0.7cm>::} 
!~-{@{-}@[|(2.5)]}
!{(-0.5,1) }*+[o][F-]{a}="a"
!{(-1,2) }*+[o][F-]{1}="b"
!{(0,0) }*+[o][F-]{c}="c"
"c" : @`{p+(2,1),c+(-2,1)} "a"
"a" : "c"
"b" : "a"
}}
+
\vcenter{\xygraph{ 
!{<0cm,0cm>;<0.7cm,0cm>:<0cm,0.7cm>::} 
!~-{@{-}@[|(2.5)]}
!{(-0.5,1) }*+[o][F-]{a}="a"
!{(-1.3,1) }*+[o][F-]{1}="b"
!{(0,0) }*+[o][F-]{c}="c"
"c" : @`{p+(2,1),c+(-2,1)} "a"
"a" : "c"
"b" : "c"
}}
 \]
and
 \[
\vcenter{ 
\xygraph{ 
!{<0cm,0cm>;<0.7cm,0cm>:<0cm,0.7cm>::} 
!~-{@{-}@[|(2.5)]}
!{(0,0) }*+[o][F-]{\star}="a"
"a" : @`{p+(2,1),c+(-2,1)} "a"
}}
\circ_\star
\vcenter{
\xymatrix@M=3pt@R=5pt@C=5pt{
*+[o][F-]{a}\ar@{->}[d]  \\
*+[o][F-]{c} 
}}=
\vcenter{
\xygraph{ 
!{<0cm,0cm>;<0.7cm,0cm>:<0cm,0.7cm>::} 
!~-{@{-}@[|(2.5)]}
!{(-0.5,1) }*+[o][F-]{a}="a"
!{(0,0) }*+[o][F-]{c}="c"
"c" : @`{p+(2,1),c+(-2,1)} "c"
"a" : "c"
}}+
\vcenter{
\xygraph{ 
!{<0cm,0cm>;<0.7cm,0cm>:<0cm,0.7cm>::} 
!~-{@{-}@[|(2.5)]}
!{(-0.5,1) }*+[o][F-]{a}="a"
!{(0,0) }*+[o][F-]{c}="c"
"c" : @`{p+(2,1),c+(-2,1)} "a"
"a" : "c"
}} \quad .
 \]
Inclusion of coloured operads $\RTM\subset\RTW$ means that every $\RTW$-algebra is, in particular, an $\RTM$-algebra, and so we may view $\RTW$-algebras as pairs $(A,M)$ of a pre-Lie algebra and its right module together with some extra structures. The simplest possible extra structure is given by the \emph{tadpole graph} consisting of one vertex and a loop edge at it:
 \[
\xygraph{ 
!{<0cm,0cm>;<0.7cm,0cm>:<0cm,0.7cm>::} 
!~-{@{-}@[|(2.5)]}
!{(0,0) }*+[o][F-]{1}="a"
"a" : @`{p+(2,1),c+(-2,1)} "a"
}
 \]
which belongs to $\RTW(\{1\},\varnothing;\m)$ and hence represents an operation from $A$ to $M$ that we shall denote by $\phi$. We note that the identity 
 \[
\vcenter{
\xygraph{ 
!{<0cm,0cm>;<0.7cm,0cm>:<0cm,0.7cm>::} 
!~-{@{-}@[|(2.5)]}
!{(0,0) }*+[o][F-]{\star}="a"
"a" : @`{p+(2,1),c+(-2,1)} "a"
}}
\circ_\star\left(
\vcenter{
\xymatrix@M=3pt@R=5pt@C=5pt{
*+[o][F-]{a}\ar@{->}[d]  \\
*+[o][F-]{c} 
}}-
\vcenter{
\xymatrix@M=3pt@R=5pt@C=5pt{
*+[o][F-]{c}\ar@{->}[d]  \\
*+[o][F-]{a} 
}}\right)
=
\vcenter{
\xymatrix@M=3pt@R=5pt@C=5pt{
*+[o][F-]{a}\ar@{->}[d]  \\
*+[o][F.]{\star} 
}}\circ_\star
\vcenter{
\xygraph{ 
!{<0cm,0cm>;<0.7cm,0cm>:<0cm,0.7cm>::} 
!~-{@{-}@[|(2.5)]}
!{(0,0) }*+[o][F-]{c}="a"
"a" : @`{p+(2,1),c+(-2,1)} "a"
}}-
\vcenter{
\xymatrix@M=3pt@R=5pt@C=5pt{
*+[o][F-]{c}\ar@{->}[d]  \\
*+[o][F.]{\star} 
}}
\circ_\star
\vcenter{
\xygraph{ 
!{<0cm,0cm>;<0.7cm,0cm>:<0cm,0.7cm>::} 
!~-{@{-}@[|(2.5)]}
!{(0,0) }*+[o][F-]{a}="a"
"a" : @`{p+(2,1),c+(-2,1)} "a"
}}
 \]
holds in $\RTW$, meaning that the tadpole graph corresponds to a linear map from $A$ to $M$ which is a $1$-cocycle for the Lie algebra $(A,[-,-])$. 

\section{The embedding theorem}\label{sec:embedding}

By contrast with the operads $\RT$ (generated by $\triangleleft$) and $\RTM$ (generated by $\triangleleft$ and $\ract$), the operad $\RTW$ is not generated by the operations $\triangleleft$, $\ract$ and $\phi$ (for instance, one can check directly that the difference of the two orientations of the $3$-cycle cannot be obtained from these operations by operadic compositions). However, it is still true that the ``obvious'' relations between $\triangleleft$, $\ract$ and $\phi$ are the defining relations of the coloured suboperad of $\RTW$ that they generate, which we prove in this section.

Let us denote by $\PLMC$ the $\{\o,\m\}$-coloured operad generated by the elements 
 \[
\triangleleft\in\PLMC(\{1,2\},\varnothing;\o),\quad \ract\in\PLMC(\{1\},\{2\};\m),\quad \phi\in\PLMC(\{1\},\varnothing;\m)
 \]
satisfying the relations described in the previous section, so that $\PLMC$-algebras are pairs $(A,M)$ with $A$ a pre-Lie algebra, $M$ is a right $A$-module, and $\phi$ is a Lie $1$-cocycle of $A$ with values in $M$.

\begin{proposition}\label{prop:combPLMC}
The operad $\PLMC$ admits the following combinatorial description.
\begin{gather*}
\PLMC(I,\varnothing;\o)=\RT(I), I\ne \varnothing,\\
\PLMC(I,\{a\};\m)=U_{\Lie}(\RT)(I), \\
\PLMC(I,\varnothing;\m)=\Omega^1_{\Lie}(\RT)\cong \bar{U}_{\Lie}(\RT)(I), I\ne \varnothing.
\end{gather*} 
Here one considers the pre-Lie algebra $\RT$ in species as a Lie algebra, and then forms the corresponding universal enveloping algebra and the module of Kähler differentials. 

The coloured operad structure corresponds to the following combinatorially define structure. The substitution of $\RT$ in the $\o$-inputs is given by the Chapoton--Livernet formulas. The substitution of $U_{\Lie}(\RT)$ and $\bar{U}_{\Lie}(\RT)$ in the $\m$-input of an operation from $\PLMC(-,\{a\};\m)\cong U_{\Lie}(\RT)(I)$ is the action of $U_{\Lie}(\RT)$ by right multiplication. 
\end{proposition}

\begin{proof}
The statement, once formulated precisely, is nearly tautological. Let us unwrap the argument step by step.

The component $\PLMC(-,\{a\};\m)$ is a right $U_{\Lie}(\RT)$-module; moreover, it is generated by the element 
$\vcenter{
\xymatrix@M=3pt@R=5pt@C=5pt{
*+[o][F.]{a} 
}}$, since the defining relations of our coloured operad define right modules without any extra properties.

The component $\PLMC(-,\varnothing;\m)$ is also a right $U_{\Lie}(\RT)$-module generated by elements $\phi(T)$ with $T\in\RT$, subject to the only relation
 \[
\phi([T_1,T_2])=\phi(T_1)\blacktriangleleft T_2-\phi(T_2)\blacktriangleleft T_1,
 \]
which identifies it with the module of Kähler differentials $\Omega^1_{\Lie}(\RT)$ for the Lie algebra associated to the pre-Lie algebra $\RT$. The latter module is precisely $\bar{U}_{\Lie}(\RT)$, and $\phi$ is under this isomorphism is identified with the universal derivation. 
\end{proof}

Let us use the previous statement to prove the main result of this section.

\begin{theorem}\label{th:emb}
The map of coloured operads $\PLMC\to\RTW$ defined by 
 \[
a_1\triangleleft a_2
\mapsto\vcenter{
\xymatrix@M=3pt@R=5pt@C=5pt{
*+[o][F-]{2}\ar@{->}[d] \\
*+[o][F-]{1} & 
}} ,
\quad
a_1\blacktriangleleft a_2
\mapsto\vcenter{
\xymatrix@M=3pt@R=5pt@C=5pt{
*+[o][F-]{2}\ar@{->}[d] \\
*+[o][F.]{1} & 
}} ,
\quad
\phi(a_1)
\mapsto\vcenter{
\xygraph{ 
!{<0cm,0cm>;<0.7cm,0cm>:<0cm,0.7cm>::} 
!~-{@{-}@[|(2.5)]}
!{(0,0) }*+[o][F-]{1}="a"
"a" : @`{p+(2,1),c+(-2,1)} "a"
}}
 \]
is injective. 
\end{theorem}

\begin{proof}
It follows from the results of Chapoton--Livernet \cite{MR1827084} and Willwacher \cite{MR4519144} mentioned in Section \ref{sec:colouredops} that this map induces isomorphisms 
\begin{gather*}
\PLMC(I,\varnothing;\o)\cong\RTW(I,\varnothing;\o),\\
\PLMC(I,\{a\};\m)\cong\RTW(I,\{a\};\m).
\end{gather*}
It remains to prove injectivity for the components $\PLMC(-,\varnothing;\m)$. According to Proposition \ref{prop:combPLMC}, such components are identified with $\bar{U}_{\Lie}(\RT)$, which, according to the Poincaré--Birkhoff--Witt theorem, can be identified with $\bar{S}(\RT_{\Lie})$. The latter is the species of non-empty sets of rooted trees, so for $|I|=n$ it is of the dimension $(n+1)^{n-1}$. Thus, the proof will be complete if we show  that the dimension of the image of $\PLMC(I,\varnothing;\m)$ is at least $(n+1)^{n-1}$.

We shall use the following combinatorial construction of ``cyclic braces''. Let us define elements $\langle a_1,\ldots,a_n\rangle\in\RTW(\{1,\ldots,n\},\varnothing;\m)$ inductively as follows. We set $\langle a_1\rangle:=\Div(a_1)$, and then let $\langle a_1,\ldots,a_{n+1}\rangle$ be given by 
 \[
\sum_{k=1}^n\langle a_1,\ldots,a_k\triangleleft a_{n+1},\ldots,a_n\rangle-\langle a_1,\ldots,a_n\rangle\blacktriangleleft a_{n+1}.
 \]
Clearly, these elements belong to the image of the map $\PLMC\to\RTW$, and moreover, it is easy to show by induction that $\langle a_1,\ldots,a_n\rangle$ is the complete symmetrization of the $n$-cycle.   

To complete the proof, note that for any rooted trees $T_1$,\ldots, $T_s$ on disjoint sets of vertices, the element $\langle T_1,\ldots,T_s\rangle$ is made of two terms, the sum over all cyclic orders of $T_1$,\ldots, $T_s$, and a sum of directed cycles of length greater than $s$. Since the element $\langle T_1,\ldots,T_s\rangle$ belongs to the image of the map $\PLMC\to\RTW$, this implies that the dimension of the image is not smaller than the number of forests of rooted trees on $n$ vertices, which is equal to $(n+1)^{n-1}$, as required. 
\end{proof}

Let us remark that the results of this section are closely related to work of Aval, Giraudo, Karaboghossian and Tanasa \cite{MR4138694} on graph insertion operads. In particular, their operad $\mathbf{LP}$ appears related to our coloured operad $\PLMC$. Our feeling, however, is that the general context of graph insertion operads leads to particularly interesting examples if one either considers coloured operads or imposes some homological degrees, like in operads that are behind the Kontsevich's graph complexes \cite{MR3348138}.

\section{Operadic approach to divergence}\label{sec:divergence}

\subsection{The action of the tadpole graph on trees}

Insertion of a tree into the only vertex of the tadpole graph defines the map 
 \[
\RTW(I,\varnothing;\o)\to \RTW(I,\varnothing;\m)
 \]
that we shall denote $\Div$ and call the \emph{divergence}. 

\begin{proposition}\label{prop:inj}
The map $\Div$ is injective.
\end{proposition}

\begin{proof}
The image of every tree under $\Div$ includes the element $\tau(T)$, where $\tau$ is the map that makes the root of a tree into a cycle of length one. Such element does not appear in $\Div(T')$ for $T'\ne T$, and therefore, all elements $\Div(T)$ are linearly independent.
\end{proof}

To obtain a map with an ``interesting'' kernel, we should therefore consider the \emph{reduced divergence} $\Div_0:=\Div-\tau$, where $\tau$ is the map defined in the previous proof. Concretely, for each rooted tree $T$, the element $\Div_0(T)$ is the sum of all ways of closing $T$ into a cycle by drawing an arrow from the root to some other vertex, so that the cycle is of length at least two. We are now ready to prove one of the main results of the paper. 

\begin{theorem}\label{th:div0}
 The kernel of the map $\Div_0$ is precisely the suboperad $\Lie\subset\RT$. 
\end{theorem}

\begin{proof}
Note that the maps $\Div$, $\tau$, and $\Div_0$ are in fact defined inside the coloured suboperad generated by 
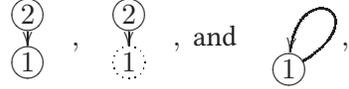
\[
\vcenter{
\xymatrix@M=3pt@R=5pt@C=5pt{
*+[o][F-]{2}\ar@{->}[d] \\
*+[o][F-]{1} & 
}} ,
\quad
\vcenter{
\xymatrix@M=3pt@R=5pt@C=5pt{
*+[o][F-]{2}\ar@{->}[d] \\
*+[o][F.]{1} & 
}} , \text{ and }
\quad
\vcenter{
\xygraph{ 
!{<0cm,0cm>;<0.7cm,0cm>:<0cm,0.7cm>::} 
!~-{@{-}@[|(2.5)]}
!{(0,0) }*+[o][F-]{1}="a"
"a" : @`{p+(2,1),c+(-2,1)} "a"
}} ,
 \]
which, according to Theorem \ref{th:emb}, is isomorphic to $\PLMC$. Thus, we can compute the kernel inside the operad $\PLMC$. In the view of Proposition \ref{prop:combPLMC}, the reduced divergence can be described as the composite
 \[
\RT\to \Omega^1_{\Lie}(\RT)\to \Omega^1_{\Lie}(\RT)/\Div(\id)U_{\Lie}(\RT)
 \]
of the universal derivation of the Lie algebra associated to the pre-Lie algebra $\RT$ and the quotient by the right $U_{\Lie}(\RT)$-submodule generated by $\Div(\id)$. What we wish to prove is that the kernel of this map is $\Lie$, that is the Lie subalgebra generated by $\id$. This is done using the following general observation going back to the work of Umirbaev \cite{MR1064070,MR1302528} (once stated, it can be proved by an easy argument relying on the Poincaré--Birkhoff--Witt theorem): if $L$ is a Lie algebra and $L'$ is a Lie subalgebra, then $L'U_{\Lie}(L)\cap L=L'$. In our case, under the identification of $\Omega^1_{\Lie}(\RT)$ with $\bar{U}_{\Lie}(\RT)$, the right $U_{\Lie}(\RT)$-submodule of $\Omega^1_{\Lie}(\RT)$ generated by $\Div(\id)$ is identified with the right ideal of $U_{\Lie}(\RT)$ generated by $\Lie$, and the abovementioned observation applies.  
\end{proof}

\subsection{Lack of volume-preserving B-series methods}\label{sec:proofIserles}

As we recalled in Section~\ref{sec:recollection}, a Butcher series method is given by a power series with terms indexed by unlabelled rooted trees. Divergence of such series takes values in the space of power series with terms indexed by directed cycles of unlabelled rooted trees, and for an individual unlabelled rooted tree, its divergence given by the sum of all ways of closing it into a cycle by drawing an arrow from the root to a vertex. For studying volume-preserving Butcher series methods, one takes the quotient by the vector subspace spanned by all directed cycles of length one; therefore, the divergence becomes the reduced divergence discussed above. One says that a Butcher series defines a volume-preserving method if its reduced divergence is zero.

A general negative result in the theory of Butcher series methods is the theorem proved independently by Iserles, Quispel and Tse in \cite{MR2334044} and by Chartier and Murua in \cite{MR2317009} that states that the only volume-preserving B-series method is the exact solution for which the flow $\tilde{f}$ coincides with $f$. This is an immediate consequence of our Theorem \ref{th:div0}, provided that one makes the following observation. Suppose that $\calA$ is a species, which we view as a certain combinatorial structure with labels from finite sets. Then the composite product of $\calA$ with the ground field $\k$ viewed as a species supported at arity zero is 
 \[
\calA\circ\k=\bigoplus_{n\ge 0}\calA(\{1,\ldots,n\})_{S_n}
 \]
is the vector space with a basis of \emph{unlabelled} $\calA$-structures. As an example, recall that Butcher series are made of linear combinations of elementary differentials, and those are indexed by the set $\RT(\bullet)$ of unlabelled rooted trees. Since $\RT$ is the underlying species of the pre-Lie operad, the remark we just made means that $\RT(\bullet)$ can be identified with the free pre-Lie algebra on one generator. Under this identification, one may view the elementary differential $F(\tau)$ as the image of $\tau$ under the homomorphism from the free pre-Lie algebra $\RT(\bullet)$ to the pre-Lie algebra $\Vect(\mathbb{R}^d)$ of all smooth vector fields on $\mathbb{R}^d$ under the pre-Lie algebra homomorphism sending $\bullet$ to $f$. 

\begin{proposition}\label{prop:volume}
The only volume-preserving B-series method is the exact flow.
\end{proposition}

\begin{proof}
Indeed, since we are considering rooted trees and cycles of rooted trees with unlabelled vertices, a volume-preserving Butcher series is in the kernel of the map 
 \[
\Div_0\colon \RT\mathop{\circ}\k\to\Cyc^+(\RT)\mathop{\circ}\k
 \]
(or rather in the completion with respect to the number of vertices, but the reduced divergence is homogeneous, so we may avoid thinking about the completion, computing the kernel degree by degree). According to Theorem \ref{th:div0}, the kernel of the map
 \[
\Div_0\colon \RT\to \Cyc^+(\RT)
 \] 
is the operad $\Lie$, and evaluating on $\k$, we get the free Lie algebra on one element, which is one-dimensional and is given by the vector space spanned by the one-vertex tree. 
\end{proof}

Note that in \cite{MR2317009}, the right-hand side of a given ODE is split into several parts, and one considers integration schemes allowing to use those parts separately. The upshot is \cite[Th.~4.10]{MR2317009} asserting that the flow must be contained in the Lie algebra generated by those parts, which is also a consequence of our operadic result, proved in the exact same way as Proposition \ref{prop:volume}. 
We also note that the proof of \cite[Th.~4.10]{MR2317009} relies crucially on results of Murua \cite{MR2271214}, many of which can be explained from the viewpoint of the coloured operad $\PLMC$.

\section{Differential graded Lie algebras and the aromatic bicomplex}\label{sec:dgla}

\subsection{Two differential graded Lie algebras in species}

Let us consider the pair of species $(\RT,\Cyc(\RT))$ with the three basic structures these species have (a Lie algebra in species, its module, and a $1$-cocycle $\phi$). It is a general fact that for a Lie algebra $\mathfrak{g}$ and a $\mathfrak{g}$-module $M$, the $1$-cocycle equation for a map $\phi\colon \mathfrak{g}\to M$ amounts to the fact that 
the operations
\begin{gather*}
[g_1+s^{-1}m_1,g_2+s^{-1}m_2]:=[g_1,g_2]+s^{-1}(g_1(m_2)-g_2(m_1)),\\
d(g+s^{-1}m)=s^{-1}\phi(g)
\end{gather*} 
defined a differential graded Lie algebra on $\mathfrak{g}\oplus s^{-1}M$. In particular, we have a differential graded Lie algebra 
 \[
\calL:=\RT\stackrel{s^{-1}\Div}{\longrightarrow}s^{-1}\Cyc(\RT)
 \]
where the nonzero Lie brackets are given by the Lie bracket on $\RT$ and the module action. To each differential graded Lie algebra $(L,d)$, one can associate its Chevalley--Eilenberg complex $C_\bullet^{CE}(L)$, so that 
 \[
C_\bullet^{CE}(L)=(\Com^c(sL),d^{CE}),
 \]
meaning that, as a species, it is the cofree cocommutative coalgebra $\Com^c(sL)$, but it additionally has a codifferential $d^{CE}=d_1^{CE}+d_2^{CE}$, where $d_1^{CE}$ is the canonical tensorial extension of $d$ from $L$ to $\Com^c(sL)$, and $d_2^{CE}$ is the extension of the map
 \[
s^{-1}[-,-]\colon S^2(sL)\to sL
 \] 
to a coderivation of $\Com^c(sL)$; if we identify $\Com^c(sL)$ with antisymmetric tensors, this unfolds to the formula
\begin{multline*}
d_2^{CE}(x_1\wedge\cdots\wedge x_n)=\\
\sum_{1\le i<j\le n}(-1)^{i+j-1}[x_i,x_j]\wedge x_1\wedge\cdots\wedge x_{i-1}\wedge x_{i+1}\wedge\cdots\wedge x_{j-1}\wedge x_{j+1}\wedge\cdots \wedge x_n , 
\end{multline*}
usually appearing in the textbook definition of Lie algebra homology \cite{MR0874337}.

\begin{theorem}\label{th:deg0}
Homology of the species $C_\bullet^{CE}(\calL)$ is concentrated in degree zero. Moreover, we have
 \[
\dim H_0^{CE}(\calL)(n)=
\begin{cases}
 (n-1)^n,\quad n\ge 2,\\ 
 \quad\,\,\, 0, \qquad\,\, \text{ otherwise,}
\end{cases}
 \]
and $H_0^{CE}(\calL)$ is identified with the species of endofunctions without fixed points. 
\end{theorem}

\begin{proof}
If one places 
 \[
S^i(s\RT)\otimes S^j(\Cyc(\RT))\subset \Com^c(s\calL)
 \]
in bi-degree $(i+j,-j)$, the maps $d_1^{CE}$ and $d_2^{CE}$ are of degree $(0,-1)$ and $(-1,0)$ respectively, making $C_\bullet^{CE}(\calL)$ into a double complex $C_{\bullet,\bullet}^{CE}$. We note that the column filtration on the direct sum total complex $\mathrm{Tot}^{\oplus}C_{\bullet,\bullet}^{CE}$ is bounded below and exhaustive, so by the Classical Convergence Theorem \cite[Th.~5.5.1]{MR1269324} the corresponding spectral sequence $\{{}^IE^r_{p,q}\}$ converges to $H_\bullet(\mathrm{Tot}^{\oplus}C_{\bullet,\bullet}^{CE})$, which is precisely the homology of the Chevalley--Eilenberg complex we are interested in.    

According to Proposition \ref{prop:inj}, the divergence map is injective on $\RT$. Moreover, it is clear that $\mathrm{coker}\Div$ has a basis of cycles of length at least two. It remains to apply the K\"unneth formula to conclude that  
 \[
{}^IE^1_{p,q}=
\begin{cases}
S^p(\mathrm{coker}\Div), \quad q=-p,\\
\qquad 0, \qquad  \,\,\, \text{ otherwise}.
\end{cases}
 \]
This immediately implies that our spectral sequence abuts on the second page, giving us a species isomorphism
 \[
H_0(\calL)\cong \Com(\mathrm{coker}\Div),
 \]
and we conclude by noting that a non-empty set of cycles of rooted trees of length at least two is precisely an endofunction without fixed points.
\end{proof} 

Examining the proof of this result, we note that it crucially relies on the property that the divergence map is injective (one can also compute the homology easily if a cocycle is surjective, see, for example, \cite[Th.~4.15]{dotsenko2024stablehomologyliealgebras}). As we saw above, there is a less trivial reduced divergence map $\Div_0$, which is also a cocycle and hence also gives rise to a differential graded Lie algebra
 \[
\tilde\calL:=\RT\stackrel{s^{-1}\Div_0}{\longrightarrow}s^{-1}\Cyc^+(\RT),
 \]
where we use the codomain $\Cyc^+(\RT)$ of cycles of length strictly greater than one. 

\begin{theorem}\label{th:deg01}
Homology of the species $C_\bullet^{CE}(\tilde\calL)$ is concentrated in degrees zero and one. Moreover, we have 
 \[
\dim H_k^{CE}(\tilde\calL)(n)=
\begin{cases}
 (n-2)^n,\quad k=0, n\ge 3,\\ 
 \quad\,\,\, 1, \qquad\,\,\,\, k=1, n=1, \\
 \quad\,\,\, 0, \qquad\,\, \text{ otherwise.}
\end{cases}
 \]
\end{theorem}

\begin{proof}
We start in the same way as in Theorem \ref{th:deg0}, placing 
 \[
S^i(s\RT)\otimes S^j(\Cyc^+(\RT))\subset \Com^c(s\calL)
 \]
in bi-degree $(i+j,-j)$, and considering the column filtration on the direct sum total complex $\mathrm{Tot}^{\oplus}C_{\bullet,\bullet}^{CE}$. 

According to Theorem \ref{th:div0}, the chain complex 
 \[
\RT\stackrel{s^{-1}\Div_0}{\longrightarrow}s^{-1}\Cyc^+(\RT)
 \]
has the homology 
 \[
\Lie\stackrel{0}{\longrightarrow}s^{-1}\mathrm{coker}\Div_0.
 \]
Moreover, in the general situation where we have a Lie algebra $\mathfrak{g}$, a $\mathfrak{g}$-module $M$, and a $1$-cocycle $\phi\colon \mathfrak{g}\to M$, there is a well defined action of the Lie algebra $\ker(\phi)$ on the quotient $M/\phi(\mathfrak{g})=\mathrm{coker}(\phi)$, and we have 
 \[
{}^IE^1_{\bullet,\bullet}\cong C_\bullet^{CE}(\Lie,S^\bullet(\mathrm{coker}\Div_0)),
 \] 
the Chevalley--Eilenberg complex computing homology of $\Lie$ with coefficients in the symmetric tensors of the module $\mathrm{coker}\Div_0$. Since the Lie algebra $\Lie$ is free, we have
 \[
{}^{I}E^2_{p,0}=H_{p}^{CE}(\Lie)=
\begin{cases}
\qquad \k, \qquad p=1,\\
\qquad 0, \qquad  \, \text{ otherwise}.
\end{cases}
 \]
Let us determine ${}^{I}E^2_{p,q}=H_{p+q}^{CE}(\Lie,S^{-q}(\mathrm{coker}\Div_0))$ for $q<0$. For that, we note that, as we remarked earlier, the image of $\Div_0$ is contained in 
 \[
\PLMC(-,\varnothing,\m)\cong \Com(\RT)\cong \Omega^1_{\Lie}(\RT),
 \] 
so it follows immediately that
 \[
\mathrm{coker}\Div_0\cong \Div(\id)U_{\Lie}(\RT)\oplus (\Cyc/\Com)(\RT).  
 \] 

Recall that Chapoton proved that for each non-empty Young diagram $\lambda$, the Schur functor $\mathbb{S}^\lambda(\RT)$ is projective as a $U_{\Lie}(\RT)$-module \cite[Prop.~4.3]{MR2682539}. In fact, using the fact that $\RT$ is free as a Lie algebra \cite[Th.~6.2]{MR2682539}, one sees  that $U_{\Lie}(\RT)$ is a free associative algebra, and hence all its projective modules (in fact, all submodules of free modules) are necessarily free \cite{MR118743}. Thus, $\mathrm{coker}\Div_0$ is a free $U_{\Lie}(\RT)$-module. Since $\Lie$ is a Lie subalgebra of $\RT$, $U_{\Lie}(\RT)$ is a free $U_{\Lie}(\Lie)$-module, and hence $\mathrm{coker}\Div_0$ is a free $U_{\Lie}(\Lie)$-module. It follows that for $q<0$ the homology of $\Lie$ with coefficients in $S^p(\mathrm{coker}\Div_0)$ is concentrated in degree zero, and, if we denote that homology by $\calM_p$,  we have
 \[
{}^{I}E^2_{p,q}=
\begin{cases}
\quad\,\,\, \calM_p, \qquad q=-p,\\
\qquad 0, \qquad  \, \text{ otherwise}.
\end{cases}
 \]
This already proves the first claim of the theorem: ${}^{I}E^2_{1,0}$ contributes to the homological degree one in $C_\bullet^{CE}(\tilde\calL)$, and all components ${}^{I}E^1_{p,-p}$, contribute to the homological degree zero in $C_\bullet^{CE}(\tilde\calL)$. 

To prove the second claim, we shall show that the exponential generating function of the Euler characteristics of $C_\bullet^{CE}(\tilde\calL)$ is equal to 
$\sum_{n\ge 1}\frac{(n-2)^n}{n!}t^n$. Since the homology of degree one is one-dimensional, and the Euler characteristic in arity one is $(1-2)^1=-1$, everything else in the Euler characteristics account for the homological degree zero. If we temporarily add the ground field as $C_0^{CE}(\tilde\calL)$, the Chevalley--Eilenberg complex becomes
 \[
C_\bullet^{CE}(\tilde\calL)\cong S(s\RT)\otimes S(\Cyc^+(\RT)),
 \]
where $S(\RT)$ is the species of forests of rooted trees, and $S(\Cyc^+(\RT))$ is the species of endofunctions without fixed points, giving the Poincaré series 
 \[
\left(1+\sum_{n\ge 1}\frac{(n+1)^{n-1}}{n!}t^n\right)^{-1}\left(1+\sum_{n\ge 1}\frac{(n-1)^n}{n!}t^n\right).
 \]
To show that this latter is equal to 
 \[
1+\sum_{n\ge 1}\frac{(n-2)^n}{n!}t^n,
 \]
we may equivalently prove that
 \[
\left(1+\sum_{n\ge 1}\frac{(n-2)^n}{n!}t^n\right)\left(1+\sum_{n\ge 1}\frac{(n+1)^{n-1}}{n!}t^n\right)=\left(1+\sum_{n\ge 1}\frac{(n-1)^n}{n!}t^n\right),
 \] 
or
 \[
(n-1)^n=\sum_{k=0}^n\binom{n}{k}(k-2)^k(n+1-k)^{n-1-k}.
 \] 
This is obtained from the formula 
 \[
(x+y)^n=\sum_{k=0}^n\binom{n}{k}x(x-(n-k)z)^{n-k-1}(y+(n-k)z)^{k}
 \] 
obtained by a change of variable in the classical Abel identity, see, for example, \cite[Exercise 5.31c]{MR4621625}, by setting $z=-1$, $x=1$, $y=n-2$. 
\end{proof}

Note that, while the dimension formula in Theorem \ref{th:deg01} resembles that of Theorem \ref{th:deg0}, qualitatively the two results are very different, in that the $S_n$-module $H_0^{CE}(\tilde\calL)(n)$ of dimension $(n-2)^n$ is not a permutation representation (i.e., it does not come from an action on a finite set). In view of Corollary \ref{cor:hom-arom} below that relates the homology of $C_\bullet^{CE}(\tilde\calL)$ to the homology of the divergence-free aromatic bicomplex, we believe that this representation deserves further study. We note that its $S_n$-character admits an elegant combinatorial formula 
 \[
\chi(H_0^{CE}(\tilde\calL)(n),\sigma)=\prod_{k=1}^n (-2+\sum_{d\mid k} da_d)^{a_k}
 \]
for each permutation $\sigma$ whose decomposition in disjoint cycles has $a_k$ cycles of length $k$ ($k=1,\ldots,n$). This can be proved similarly to the dimension formula above, using the known character formulas for endofunctions and forests \cite{MR1038365,MR927766} and the Abel identity. 

\subsection{The aromatic bicomplex and the Chevalley--Eilenberg complex}

We already saw in Section \ref{sec:proofIserles} that for a Schur functor $\calA$, the composite product $\calA\circ\k$ is the vector space with a basis of unlabelled $\calA$-structures. To account for the presence of black and white vertices in the aromatic bicomplex, we shall consider the composite product with the acyclic two-term complex $K_\bullet:=\k\to\k$. Since the symmetry isomorphisms in the symmetric monoidal category of chain complexes use the Koszul rule of signs, the composite product can be computed by the formula  
 \[
\calA\circ K_\bullet=\bigoplus_{n\ge 0}\calA(\{1,\ldots,n\})^{S_n}\bigoplus_{p+q=n}\calS(\{1,\ldots,n\})\otimes_{S_p\times S_q}(\mathrm{triv}_p\otimes\mathrm{sgn}_q),
 \]
giving the vector space with a basis of two-coloured $\calA$-structures whose black labels are indistingushable black labels and whose white labels are distinguishable white labels but the elements are skew-symmetric with respect to permutations of white labels. Moreover, this vector space comes with a differential $d^K_*$ which is induced by the differential $d^K$ of $K_\bullet$; it is a sum of all possible ways to replace a black label with a new white label. We are now ready to relate the complexes discussed in the previous section to the aromatic bicomplex. 

\begin{proposition}
We have the following isomorphisms of bicomplexes:
\begin{gather*}
(\Omega_{\bullet,\bullet},d^H,d^V)\cong (\Com^c(s\calL)(K_\bullet),d^{CE},d^K_*),\\
(\tilde\Omega_{\bullet,\bullet},d^H,d^V)\cong (\Com^c(s\tilde\calL)(K_\bullet),d^{CE},d^K_*).
\end{gather*}
\end{proposition}

\begin{proof}
Let us start with the full aromatic bicomplex. As we mentioned above, considering $\Com^c(s\calL)(K_\bullet)$ amounts to considering graphs with several indistinguishable black vertices and several distinguishable white vertices with respect to which the graphs are anti-symmetric. This means that the underlying vector space of $\Com^c(s\calL)(K_\bullet)$ is precisely what we require: the anti-symmetrization over white vertices comes from $K_\bullet$ and the anti-symmetrization over trees comes from the suspension $s\RT$ in 
 \[
\Com^c(s\calL)=\Com^c(s\RT\to\Cyc(\RT)).
 \]  
The differential of $K_\bullet$ clearly induces the vertical differential~$d^V$. Moreover, the differential $d^{CE}$ does correspond to $d^H$. To see that, we note that the latter computes all ways of connecting the root to some vertex; connecting it to a vertex of another tree amounts to computing the pre-Lie product, of which the anti-symmetrization gives the Lie bracket of rooted trees, connecting it to a vertex of a cycle of rooted trees amounts to computing the right pre-Lie (that is, Lie) action of trees on cycles of rooted trees, which also contributes to the Lie bracket of $\calL$, and connecting it to a vertex of the same tree amounts to computing the divergence which corresponds to the differential of $\calL$. 
The argument for the divergence-free aromatic bicomplex is analogous: passing to the quotient $\tilde\Omega_{\bullet,\bullet}$ corresponds to replacing the species $\Cyc(\RT)$ by the species $\Cyc^+(\RT)$, and all other steps of the proof remain the same.    
\end{proof}

We use this statement to obtain new proofs of some of the key results of \cite{MR4624837}, namely Theorem 2.9 and Theorem 2.10 of that paper. 

\begin{corollary}\label{cor:hom-arom}\leavevmode
\begin{enumerate}
\item Homology of the differential $d^H$ of the aromatic bicomplex is concentrated in bi-degrees $(0,\bullet)$.
\item Homology of the differential $d^H$ of the divergence-free aromatic bicomplex is concentrated in bi-degrees $(0,\bullet)$ and $(1,\bullet)$; in the latter degrees, the homology is spanned by the one-vertex trees (with either black or white vertex).
\item Homology of the differential $d^V$ of both the aromatic bicomplex and the divergence-free aromatic bicomplex is zero.
\end{enumerate}
\end{corollary}

\begin{proof}
All of these results amount to applying the K\"unneth formula for symmetric sequences~\cite[Prop.~6.2.5]{MR2954392} to 
 \[
\Com^c(s\calL)(K_\bullet)=\Com^c(s\calL)\circ K_\bullet\text{ and } \Com^c(s\tilde\calL)(K_\bullet)=\Com^c(s\tilde\calL)\circ K_\bullet
 \]  
Indeed, Theorem \ref{th:deg0} implies that we have
\begin{multline*}
H_\bullet(\Omega_{\bullet,\bullet},d^H)\cong H_\bullet(\Com^c(s\calL)\circ K_\bullet,d^{CE})\cong\\ 
H_\bullet(\Com^c(s\calL),d^{CE})\circ K_\bullet\cong H_0(\Com^c(s\calL),d^{CE})\circ K_\bullet
\end{multline*}
and hence is concentrated in bi-degrees $(0,\bullet)$. Furthermore, we have 
 \[
H_\bullet(\Omega_{\bullet,\bullet},d^V)\cong 
H_\bullet(\Com^c(s\calL)\circ K_\bullet, d^K_*)\cong \Com^c(s\calL)\circ H_\bullet(K_\bullet,d^K)=0.
 \]
Similarly, Theorem \ref{th:deg01} implies that we have
 \[
H_\bullet(\tilde\Omega_{\bullet,\bullet},d^H)\cong H_\bullet(\Com^c(s\tilde\calL)\circ K_\bullet,d^{CE})\cong
H_\bullet(\Com^c(s\tilde\calL),d^{CE})\circ K_\bullet,
 \]
and hence the homology is concentrated in bi-degrees $(0,\bullet)$ and $(1,\bullet)$. Moreover, $H_1(\Com^c(s\tilde\calL),d^{CE})$ is supported at arity one and is one-dimensional, implying the claim about one-vertex trees. 
In the same way, we have 
 \[
H_\bullet(\tilde\Omega_{\bullet,\bullet},d^V)\cong 
H_\bullet(\Com^c(s\tilde\calL)\circ K_\bullet,d^K_*)\cong \Com^c(s\tilde\calL)\circ H_\bullet(K_\bullet,d^K)=0.
 \]
\end{proof}

We recall that the second part of this latter result is the key step in a complete classification of volume-preserving aromatic Butcher series methods, see \cite[Th.~4.17]{MR4624837}. 

\section*{Acknowledgements} The first author is indebted to Yvain Bruned who asked him about possible operadic context for aromatic Butcher series and gave useful feedback on the draft version of the paper and to Guodong Zhou for hospitality at ECNU in Shanghai where a large part of this paper was written.  
We gratefully acknowledge useful discussions of the aromatic bicomplex we had with Adrien Laurent during the workshop of ANR CARPLO in Besse in October 2024. We also thank Dominique Manchon for organizing that event. 

\section*{Funding}

The first author is supported by the ANR project HighAGT (ANR-20-CE40-0016), and by Institut Universitaire de France. The second author is funded by a postdoctoral fellowship of the ERC Starting Grant ``Low Regularity Dynamics via Decorated Trees'' (LoRDeT) of Yvain Bruned (grant agreement No.\ 101075208).

\printbibliography
\appendix

\section{Further consequences of Theorem \ref{th:div0}}\label{app:conseqdiv0}

\subsection{Lie elements in free pre-Lie algebras}

Let $\PL(V)$ be the free pre-Lie algebra generated by $V$. Markl proved in \cite{MR2325698} that the Lie subalgebra 
 \[
\Lie(V)\subset\PL(V)
 \] 
can be described explicitly as the kernel of a certain map 
$\PL(V)\to\PL(V\oplus\k)$.  
Our work suggests the following other criteria for Lie elements in free pre-Lie algebras. 

\begin{proposition}
The space of Lie elements in the free pre-Lie algebra $\PL(V)$ is equal to the kernel of each of the maps
\begin{gather*}
\Div_0\colon\RT(V)\to\Cyc(\RT)(V),\\
\mathrm{Sym}\circ\Div_0\colon\RT(V)\to\Com(\RT)(V).
\end{gather*}
It is also the space of all elements $f$ such that  
 \[
\pi\d_{\Lie}(f)=0,
 \]
where $\d_{\Lie}$ is the universal derivation $\RT(V)\to \Omega^1_{\Lie}(\RT(V))$, and $\pi$ is the canonical projection to the quotient of the $U_{\Lie}(\RT(V))$-module $\Omega^1_{\Lie}(\RT(V))$ by the submodule generated by $\d_{\Lie}(V)$. 
\end{proposition}

\begin{proof}
The first criterion is immediate from Theorem \ref{th:div0}. The second one is a literal translation of the key step in the proof of that theorem.
\end{proof}

\subsection{Reduced divergence and operadic twisting}

Let us show that our Theorem \ref{th:div0} gives a slight improvement of a known result about the operad $\PL$. 
In~\cite{MR4709210}, the first author and Khoroshkin computed the homology of the differential graded operad $\Tw(\PL)$, the operadic twisting of the  operad $\PL$ in the sense of Willwacher \cite{MR3348138}. The easiest way to describe that operad and its differential follows the description of \cite{MR4621635} and is as follows. Its  arity $n$ component $\Tw(\PL)(n)$ is spanned by rooted trees with ``normal'' vertices labelled $1$, $\ldots$, $n$ and a certain number of ``special'' vertices labelled~$\alpha$. The differential $d_{\Tw}$ sends each tree~$T$ to the sum of three terms:

\begin{itemize}
\item[-] the sum over all possible ways to split a normal vertex into a normal one retaining the label and a special one, and to connect the incoming edges of that vertex to one of the two vertices thus obtained, so that the term where the vertex further from the root retains the label is taken with the plus sign, and the other term is taken with the minus sign,
\item[-] the sum over all special vertices of $T$ of all possible ways to split that vertex into two special ones, and to connect the incoming edges of that vertex to one of the two vertices thus obtained, taken with the plus sign,
\item[-] grafting the tree $T$ at the new special root, taken with the minus sign, and the sum of all possible ways to create one extra special leaf, taken with the plus sign.
\end{itemize}

It is established in \cite[Th.~5.1]{MR4709210} that the homology of $\Tw(\PL)$ is concentrated in homological degree zero and is isomorphic to the operad $\Lie$. It turns out that our Theorem \ref{th:div0} is related to this result. Namely, the differential $d_{\Tw}$, when evaluated on a tree without special vertices, can be written as a sum of two parts, the terms where the special vertex is the root and all other terms. Trees where the special vertex is the root may be identified with $\bar{U}_{\Lie}(\RT)$. Under this identification as well as the identification $\PLMC(I,\varnothing;\m)\cong \bar{U}_{\Lie}(\RT)(I)$ of Proposition \ref{prop:combPLMC}, the first part of the differential $d_{\Tw}$ is precisely the reduced divergence. Thus, the statement that the kernel of the reduced divergence coincides with the operad $\Lie$ may be viewed as a slightly more precise form of the statement that the kernel of $d_{\Tw}$ coincides with the operad $\Lie$.

\section{A version of our results for graph complexes}\label{app:graphcx}

In \cite{MR3590540}, questions similar to the ones discussed in this paper are discussed in the case of graphs that are not directed. In that case, there are of course no constraints on the combinatorics of graphs: that combinatorics is constrained to trees-and-cycles only when graphs are directed and every vertex has at most one outgoing edge. We shall now discuss an operadic version of some of the results of \cite{MR3590540} which exhibit an interesting parallel with our Theorem \ref{th:div0} and Theorem \ref{th:deg01}.

To be precise, we consider the species $\Graphs$ made of linear combinations of graphs with edges of homological degree $-1$ (note that this prohibits graphs from having multiple edges). We denote by $\CGraphs$ the subspecies made of linear combinations of connected graphs. In our context, there is no natural way to separate graphs by types, and so we consider the single-colour operad $\CGraphs$ with the operad structure given by the same graph insertion as for $\RT$ and $\RTW$. The graphs 
 \[
\vcenter
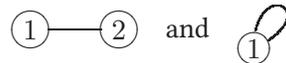
{
\xymatrix@M=4pt@R=4pt@C=20pt{
*+[o][F-]{1}\ar@{-}[r] & *+[o][F-]{2} \\
}} \quad\text{and}\quad 
\vcenter{\xygraph{ 
!{<0cm,0cm>;<0.5cm,0cm>:<0cm,0.5cm>::} 
!~-{@{-}@[|(2.5)]}
!{(0,0) }*+[o][F-]{1}="a"
"a" - @`{p+(2,1),c+(-2,1)} "a"
}} 
 \]
are known \cite{MR4726567} to generate a suboperad of $\CGraphs$ isomorphic to the operad $\calS\Delta\Lie$ encoding shifted Lie algebras equipped with a unary operator of homological degree $-1$ that squares to zero and is a derivation of the shifted Lie bracket, so $\CGraphs$ becomes a shifted dg Lie algebra in species. The graphs that have at least one loop at their vertices form a dg ideal of that dg Lie algebra, and we denote by $\CGraphs^r$ the quotient by that ideal; it is spanned by connected graphs without loops at their vertices. It has its Chevalley--Eilenberg complex 
 \[
 C_\bullet^{CE}(\CGraphs^r)=(\Com^c(\CGraphs^r),d^{CE}=d^{CE}_1+d^{CE}_2),
 \]
where $\Com^c(\CGraphs^r)$ may be naturally identified with the species of all reduced graphs $\Graphs^r$ (note that there is no suspension of $\CGraphs^r$ since we work with shifted Lie algebras), and $d^{CE}$ is the corresponding Chevalley--Eilenberg differential. Moreover, we can identify the \emph{quotient} complex $C_\bullet^{CE}(\CGraphs^r)$ with a \emph{sub}complex of $C_\bullet^{CE}(\CGraphs)$ if we consider all reduced graphs with the differential 
 \[
\left[ 
\vcenter{\xygraph{ 
!{<0cm,0cm>;<0.5cm,0cm>:<0cm,0.5cm>::} 
!~-{@{-}@[|(2.5)]}
!{(0,0) }*+[o][F-]{1}="a"
"a" - @`{p+(2,1),c+(-2,1)} "a"
}},-\right] \quad,
 \]
given by the operadic commutator with the tadpole graph: indeed, the operadic substitution into the tadpole graph is the sum of all possible ways of add one edge, and the commutator with the tadpole graph is the difference of that of all possible substitutions of the tadpole in vertices of the graph, giving precisely the sum of all possible ways of add one edge between two distinct vertices.

Recall that \cite[Prop.~2]{MR3590540} asserts that the homology of the shifted Chevalley--Eilenberg complex $C_\bullet^{CE}(\CGraphs^r)\circ\k$ of unlabelled graphs is one-dimensional and represented by the one-vertex graph. Let us record here an operadic version of this result, which, coincidentally, is almost completely identical to the unlabelled version.

\begin{proposition}
The complex $C_\bullet^{CE}(\CGraphs^r)$ has one-dimensional homology supported at arity one. 
\end{proposition}

\begin{proof}
Indeed, this complex admits a very simple contracting homotopy $h$ given by the sum, with signs, of all ways to remove one edge from the graph: on the space of $n$-vertex graphs this map satisfies $d^{CE}h+hd^{CE}=\binom{n}{2}\id$, and $\binom{n}{2}\ne 0$ unless $n=1$. 
\end{proof}

Note that the vector space spanned by connected unlabelled graphs is a subcomplex of the full complex $C_\bullet^{CE}(\CGraphs^r)\mathop{\circ}\k$. Recall that \cite[Cor.~3]{MR3590540} the homology of that latter subcomplex is two-dimensional, represented by one- and two-vertex graphs. In this case, the operadic result, reminiscent of our Theorem~\ref{th:div0}, is more subtle.

\begin{proposition}
We have
 \[
H_\bullet(\CGraphs^r,d^{CE})\cong\calS\Lie.
 \]
\end{proposition}   

\begin{proof}
We shall prove the statement by induction on arity of elements. It turns out that the argument of \cite[Cor.~3]{MR3590540} can be simplified in a way that then allows us to adapt it for our purposes. For arity $1$, the statement is trivial. Suppose we proved it in arities not exceeding $n$, and consider an element $\gamma$ that is a linear combination of connected graphs on $n$ labelled vertices satisfying $d^{CE}\gamma=0$. We mentioned above that in the complex of not necessarily connected graphs, there is a contracting homotopy $h$ that removes an edge in all possible ways, so that $d^{CE}h\gamma=\binom{n}{2}\gamma$, and $\gamma$ is the image under $d^{CE}$ of the element $\binom{n}{2}^{-1}h\gamma$ that is a combination of connected graphs and graphs with two connected components; we write 
 \[
\binom{n}{2}^{-1}h\gamma=\nu_1+\nu_2
 \]
according to that distinction. If $\nu_2=0$, then $\gamma$ is in the image of $d^{CE}$ in the complex of connected graphs and so does not contribute to the homology. If $\nu_2\ne 0$, we argue as follows. We represent $d^{CE}=d^{CE}_c+d^{CE}_r$, where $d^{CE}_c$ does not change the number of connected components and $d^{CE}_r$ reduces the number of connected components by one. We must have $d^{CE}_c(\nu_2)=0$. However, the complex of graphs with two connected components equipped with the differential $d^{CE}_c$ is the symmetric square of the complex of connected graphs, and by induction we may write
 \[ 
\nu_2=\ell+d^{CE}_c(\nu_2'),
 \] 
where $\ell$ is a linear combination of products of shifted Lie elements. Then 
 \[
\binom{n}{2}^{-1}h\gamma=\nu_1+\ell+d^{CE}_c(\nu_2')=\nu_1+\ell+d^{CE}(\nu_2')-d^{CE}_r(\nu_2')=\nu_1'+\ell+d^{CE}(\nu_2'),
 \]
where we absorbed new connected elements into $\nu_1'$. Now we have 
 \[
\gamma=\binom{n}{2}^{-1}d^{CE}h\gamma=d^{CE}\nu_1'+d^{CE}\ell.
 \] 
It remains to note that $d^{CE}\ell=d^{CE}_r(\ell)$ is a shifted Lie element, so we proved that $\gamma$ is homologous to a shifted Lie element. Also, these elements have the smallest possible number of edges for a connected graph, so they do not vanish in the homology.
\end{proof}

\section{Another approach to the operad \texorpdfstring{$\PLMC$}{PLMC}}\label{app:colouredGB}

In this section, we explain how one could discover guess the result of Proposition \ref{prop:combPLMC} using Gr\"obner bases for coloured operads and Koszul duality. It is also possible to use results of this section to interpret homology calculations of Section~\ref{sec:dgla} as homology of deformation complexes of certain maps of coloured operads; this viewpoint will be discussed elsewhere.

\begin{proposition}
The coloured operad $\PLMC$ is Koszul. Its Koszul dual is the $\{\o,\m\}$-coloured operad generated by the elements
 \[
\circ\in\PLMC^!(\{1,2\},\varnothing;\o),\quad \bullet\in\PLMC^!(\{1\},\{2\};\m),\quad\delta\in\PLMC^!(\{1\},\varnothing;\m)
 \]
such that 
\begin{gather*}
(a\circ b)\circ c=(a\circ c)\circ b=a\circ(b\circ c)=a\circ(c\circ b),\\
(a\bullet b)\bullet c=(a\bullet c)\bullet b=a\bullet(b\circ c)=a\bullet(c\circ b),\\
\delta(a\circ b)=\delta(a)\bullet b=\delta(b)\bullet a.
\end{gather*}
\end{proposition}

\begin{proof}
The presentation of the Koszul dual operad is obtained by a simple direct computation. To establish that our operad is Koszul, we shall argue as follows. First, the presentation of the Koszul dual operad implies that we have
\begin{gather*}
\dim\PLMC^!(\{1,\ldots,n\},\varnothing;\o)=n, n\ge 1,\\
\dim\PLMC^!(\{1,\ldots,n\},\{a\};\m)=1, n\ge 0,\\
\dim\PLMC^!(\{1,\ldots,n\},\varnothing;\m)=1, n\ge 1,
\end{gather*} 
since the $\o$-coloured part is the operad $\Perm=\PL^!$, and the other components are at most one-dimensional due to the relations that are satisfied, and at the same time at least one-dimensional since operad all our relations are differences of monomials in the free coloured operad, and the same is true for every pre- and post-composition of those relations with monomials, so no monomial vanishes individually in the quotient. 

We finish the argument using Gr\"obner bases for coloured operads \cite{MR4375008}. We shall work with the generators 
 \[
[a_1,a_2],\quad  a_1\cdot a_2=a_1\triangleleft a_2+a_2\triangleleft a_1,\quad a_1\blacktriangleleft a_2, \quad \phi(a_1)
 \] 
of the coloured operad $\PLMC$. For a suitable order of coloured shuffle tree monomials in the corresponding free coloured operad, the monomials representing the terms 
 \[
[a_1,[a_2,a_3]],\quad a_1\cdot[a_2,a_3],\quad a_1\cdot(a_2\cdot a_3),\quad
a_1\blacktriangleleft [a_2,a_3],\quad
\phi([a_1,a_2])
 \]
are the leading terms of the relations of the coloured operad $\PLMC$; for this, it is useful to recall that the defining relations of the operad $\PL$ in terms of the operations $[a_1,a_2]$ and $a_1\cdot a_2$ are the Jacobi identity and
\begin{multline*}\label{eq:relPL}
(a_1\cdot a_2)\cdot a_3 - a_1\cdot (a_2 \cdot a_3) - a_1\cdot [a_2, a_3] - [a_1, a_2]\cdot a_3 \\- 2[a_1, a_3]\cdot a_2+
 [a_1, a_2\cdot a_3] + [a_1\cdot a_2, a_3] + [[a_1, a_3], a_2] = 0 ,
\end{multline*}
see \cite[Lemma~2]{MR3203367}. Moreover, according to \cite[Prop.~4]{MR3203367}, these monomials are precisely the quadratic normal forms for the coloured operad $\PLMC^!$ with respect to the opposite order. Monomials whose quadratic divisors are these normal forms give an upper bound on the dimensions of components of the operad $\PLMC^!$, and that bound is sharp if and only if our quadratic relations form a Gr\"obner basis. Since we manifestly obtain the correct dimensions for the components $\PLMC^!(\{1,\ldots,n\},\{a\};\m)$ and $\PLMC^!(\{1,\ldots,n\},\varnothing;\m)$, the operad $\PLMC^!$ has a quadratic Gr\"obner basis and is therefore Koszul.
\end{proof}

\begin{remark}
There are several different choices of a Gr\"obner basis that can be used in the second half of the proof; it is not even necessary to change to the generators $[a_1,a_2]$ and $a_1\cdot a_2$. The advantage of this choice is that it at the same time proves, using \cite[Th.~4]{MR3203367}, that $\PL$ is free as a left $\Lie$-module (so that free Lie algebras are free as Lie algebras), and, moreover, it can be used to re-prove Theorem \ref{th:div0} without the explicit description of Proposition \ref{prop:combPLMC}, just by appealing to normal forms with respect to a Gr\"obner basis.
\end{remark}

The following result, and especially the ending of its proof, leads to an educated guess of the exact statement of Proposition \ref{prop:combPLMC}.

\begin{corollary}\label{cor:dimPLMC}
We have
\begin{gather*}
\dim\PLMC(\{1,\ldots,n\},\varnothing;\o)=n^{n-1}, n\ge 1,\\
\dim\PLMC(\{1,\ldots,n\},\{a\};\m)=(n+1)^{n-1}, n\ge 0,\\
\dim\PLMC(\{1,\ldots,n\},\varnothing;\m)=(n+1)^{n-1}, n\ge 1.
\end{gather*} 
\end{corollary}

\begin{proof}
The previous proof implies the following Hilbert--Poincaré series of the Koszul dual operad:
 \[
f^{\o}_{\PLMC^!}(x,y)=xe^x,\quad f^{\m}_{\PLMC^!_\m}(x,y)=ye^x+e^x-1. 
 \]
Since the operad $\PLMC$ is Koszul, the functional equation for the Hilbert--Poincaré series \cite[Th.~3.3.2]{MR1301191}  tells us that
\begin{gather*}
f^{\o}_{\PLMC}(x,y)e^{-f^{\o}_{\PLMC}(x,y)}=x,\\ 
f^{\m}_{\PLMC}(x,y)e^{-f^{\o}_{\PLMC}(x,y)}+e^{-f^{\o}_{\PLMC}(x,y)}-1=y,
\end{gather*}
which of course recovers the fact that $f^{\o}_{\PLMC}(x,y)=f_{\RT}(x)$ that we already know, and 
 \[
f^{\m}_{\PLMC}(x,y)=ye^{f_{\RT}(x)}+e^{f_{\RT}(x)}-1,
 \]
instantly proving the necessary result, if one notes that the exponential of the generating function for labelled rooted trees is the generating function for forests of labelled rooted trees
 \[
\sum_{n\ge 0} \frac{(n+1)^{n-1}}{n!}x^n.
 \] 
\end{proof}

\end{document}